\documentclass[draft]{amsart}

\usepackage{amssymb}
\usepackage{url}
\usepackage{amscd}
\usepackage{enumerate}
\usepackage[all]{xy}

\def \N{{\mathbb N}}

\def \R{{\mathbb R}}

\def \1{{\mathbb 1}}

\theoremstyle{plain}
\newtheorem{theorem}{Theorem}
\newtheorem{proposition}{Proposition}
\newtheorem{definition}{Definition}
\newtheorem{lemma}{Lemma}

\theoremstyle{remark}
\newtheorem{remark}{Remark}

\newtheorem{Exemp}{Example}

\begin{document}
\title[Nonlinear differential equations]{Nonlinear differential equations in abstract  Banach subspace of $BC(\R)$.}
\author{Mohammed Bachir$^*$, Haifa Ben Fredj$^\dagger$}
\date{14/11/2021}
\address{$^*$ Laboratoire SAMM 4543, Universit\'e Paris 1 Panth\'eon-Sorbonne, France}
\email{Mohammed.Bachir@univ-paris1.fr}

\address{$^\dagger$ MaPSFA, ESSTHS, University of Sousse-Tunisia}
\email{haifabenfredj@essths.u-sousse.tn}

\begin{abstract} We prove results of  existence of a solution (resp. existence and uniqness of a solution) for nonlinear differential equations of type $x'(t) +G(x,t) x(t) = F(x,t),$ in an abstract Banach subspace $X$ of the space of bounded real-valued continuous functions, satisfying some general and natural property. In our work, the functions $F$ and $G$ jointly depend on the variables $(x,t)\in X\times \R$. Several examples will be given, in various function spaces, to illustrate our results. The vector-valued framework is also considered.
\end{abstract}

\maketitle

\noindent {\bf Keyword, phrase:} Nonlinear differential equations, Global attractivity, Periodic and pseudo almost periodic functions, Schauder and Banach-Picard fixed point theorems, Spaces of real and vector-valued continuous functions.

\section{Introduction}
Let $X$ be a Banach subspace of the Banach space  $(BC(\R),\|\cdot\|_{\infty})$  of all real-valued bounded continuous  functions equipped with the sup-norm and  let $F, G :X\to X$  be two functions. The goal of this paper is  to give existence of solutions $x\in X$ (under some conditions on the space $X$ and the functions $F$ and $G$) to the first order differential equations of type:
$$(E) \hspace{3mm} x'(t) +G(x,t) x(t) = F(x,t), \hspace{2mm} \forall  t\in \R$$
where, $F(x,t):=F(x)(t)$ and the same for $G$.
We prove in our main result (Theorem \ref{existence}) an existence result by using the Schauder's fixed point theorem and another result of existence and uniqueness by using the Banach-Picard theorem.
Our formalism encompasses several examples of spaces $X$ of $BC(\R)$ in particular the space $PAP(\R)$ of all pseudo almost periodic functions (we will give the precise definition in Section \ref{Sexemple}) which has attracted the interest of many authors in recent years. For references in this aera, we refer for example to the non-exhaustive list of works \cite{AC, BEFDB, JBGP, DLN, SLGX, Lb, LHY,LS, WXD}. Our contributions in this paper are axed  in the following directions:

--- The equations treated in the litterature in the $PAP(\R)$ space (see for instance the above references), are of the the form $x'(t)+a(t)x(t)=H(x(t),t)$, where $a$ is a function depending only on $t$ and $H$ is a function defined on $\R\times \R$. In other words, the map $x\mapsto a(\cdot)x$ is a linear operator.  In our results, we deal with operators not necessarily linear, that is, we replace $a(t)$ with a more general function $G(x,t)$ depending on both $x$ and $t$. Moreover,  the function $H$ is also replaced by a more general function depending on $x$ and $t$. For example  for any $\alpha, \beta\in L^1(\R)$  such that $0<\|\alpha\|_1\leq 1$ (where $\|x\|_1=\int_{-\infty}^{+\infty}|x(s)|ds$ and $x*\alpha(t)=\int_{-\infty}^{+\infty} x(s)\alpha(s-t) ds$ denotes the convolution of $x$ and $\alpha$), the pseudo almost periodic functions
\begin{eqnarray*}
F(x,t)&=&\frac{1}{3}(sin(t)+sin(\sqrt{2}t)+ (1+t^2)^{-1}x*\alpha(t))\\
G(x,t)&=&3+sin(2t)+(1+t^2)^{-1}cos(x*\beta(t)),
\end{eqnarray*}
 can not be writen in the form $H(x(t),t)$ for some  function $H$ defined on $\R\times \R$. However, with these functions, our results apply and give  at last one pseudo almost periodique solution to the equation $(E)$ (see examples  in Section \ref{Sexemple} for  details). 

--- Our approach unifies several spaces of functions. We deal with abstract Banach subspace $X$ of $BC(\R)$ including some classical spaces as  the space of all continuous $w$-periodic functions, the space of all almost periodic functions or the space all pseudo almost periodic functions (see Proposition \ref{algebra}). Thus, depending on the type of functions $F$ and $G$, the solutions will exist in the adequat corresponding type of space (see examples in Section \ref{Sexemple}). It is the interest of working in an abstract subspace $X$ of $BC(\R)$ satisfying a natural condition that we call $(H_0)$ in this paper. Moreover, our result hold also for Banach subspaces $X$ of vector-valued functions, that is, Banach spaces of functions from $\R$ into a Banach space $E$ of finite dimension. 
\vskip5mm
This paper is organized as follows. In section \ref{S1}, we give our first main result of existence of solutions of the equation $(E)$ under some general hypothesis (Theorem \ref{existence}) and we will then give several examples to illustrate this result. In Section \ref{Global} we give our second main theorem consisting on the attractivity of solutions (Theorem \ref{attractivity}). Finally, in Section \ref{Vector}, we  will discuss the case of vector valued function spaces.
\section{The main result}  \label{S1}
Let $X$ be a Banach subspace of $BC(\R)$. The set $B_X(0,r)$ denotes the closed ball of $X$ centred at $0$ with radious $r>0$. For each $l,r>0$, we define the following closed convex subsets of $X$:  $$B_{[l,r]}:=\lbrace x\in X: x(t)\in [l,r], \forall t\in \R\rbrace,$$
$$X_{[l,+\infty[}:=\lbrace x\in X: x(t)\in [l,+\infty[, \forall t\in \R\rbrace,$$
We need to introduce, the following well defined operator for each $l>0$ (see Lemma \ref{Contint}):
\begin{eqnarray*}
T: BC(\R) \times BC(\R)_{[l,+\infty[}&\to& BC(\R)\\
(f,g) &\mapsto& [t\mapsto \int_{-\infty}^t e^{-\int_s^t g(u)du} f(s) ds],
\end{eqnarray*}
It is classical and easy to see that for every $(f,g)\in  BC(\R)\times BC(\R)_{[l,+\infty[}$, the function $T(f,g)$ is  differentiable and satisfies:
\begin{eqnarray}\label{fixe}
T(f,g)'(t) &=&-g(t)T(f,g)(t)+f(t), \hspace{2mm} \forall t\in \R.
\end{eqnarray}
\vskip5mm

We consider  the following conditions $(H_0)$, $(H_1)$ and $(\tilde{H}_1)$:

$(H_0)$ the subspace $X$ is invariant under $T$ in the sens that for each $l>0$,  $T(X \times X_{[l,+\infty[})\subset X$.

The property $(H_0)$ is satifed by several classical Banach subspaces of $BC(\R)$, see examples in Section \ref{Sexemple}.

$(H_1)$ The functions  $F, G :(X,\|\cdot\|_{\infty}) \to (X,\|\cdot\|_{\infty})$ are continuous and satisfies: 

$\bullet$ $\inf_{x\in X, t\in \R}G(x,t) >0$ and  there exists $k, M\in \R$ such that $F$ and $G$ are bounded on $B_{[k,M]}$ and that $k\leq \frac{F(x,t)}{G(x,t)}\leq M$ for all $(x,t)\in B_{[k,M]}\times \R$.

$\bullet$ for every sequence $(x_n)\subset X$, if $(x_n)$ converges on each compact of $\R$ in $BC(\R)$, then $(F(x_n))$ and $(G(x_n))$ are relatively compact in $(X,\|\cdot\|_{\infty})$.
\vskip5mm
 Notice that if we assume that $F$ and $G$ are bounded on the whole space $X$ and $\inf_{x\in X, t\in \R}G(x,t) >0,$ then we can take  in the hypothesis $(H_1)$
$$k:=  \inf_{(x,t)\in X\times \R} \frac{F(x,t)}{G(x,t)} \textnormal{ and } M:= \sup_{(x,t)\in X\times \R} \frac{F(x,t)}{G(x,t)}.$$

Notice also that the second point of $(H_1)$ is crucial and ensures that the Schauder fixed point theorem applies. This condition is automaticaly satisfied in the subspace $X=C_w(\R)$ of all $w$-periodic continuous functions since in this case the uniform convergence on each compact of $\R$ is equivalent to the uniform convergence on $\R$. It is also satisfied on other spaces, in general and various situations (see Section \ref{Sexemple} for some examples) and has already been used in the literature (see for instance the condition  $(H_3)$  in \cite{SLGX} and the condition $(E5)$, page 248 in \cite{Dt}).

$(\tilde{H}_1)$ The functions  $F, G :(X,\|\cdot\|_{\infty}) \to (X,\|\cdot\|_{\infty})$ are Lipschitz, $$l:=\inf_{x\in X, t\in \R}G(x,t) >0,$$
and there exists $k, M\in \R$ such that, $k\leq \frac{F(x,t)}{G(x,t)}\leq M$ for all $(x,t)\in B_{[k,M]}\times \R$, $$r:=\sup_{x\in B_{[k,M]}}\|F(x)\|_{\infty}<+\infty$$
 and  $$\max(\frac{r}{l^2}, \frac{1}{l})(L_F+L_G)<1,$$
where $L_F$ and $L_G$ denotes the constant of Lipschitz of $F$ and $G$ respectively.
\vskip5mm
 The hypothesis $(H_1)$ and $(\tilde{H}_1)$ are easily satisfied by several examples (see the examples in Section \ref{Sexemple}).

Notice  that the set $B_{[k,M]}$ is a closed convex subset of $X$.

\begin{theorem} \label{existence} Under the hypothesis $(H_0)$ and $(H_1)$ (resp. $(H_0)$ and $(\tilde{H}_1)$), the equation $(E)$  has at least one  solution $x^*$ in $B_{[k,M]}$  (resp. has a unique solution $x^*$ in $B_{[k,M]}$), where $k$ and $M$ are given by the hypothesis $(H_1)$.
\end{theorem}
The proof of the above theorem will be given in the following section. 
\begin{remark} If moreover the function $F$  is assumed to be a positive function, then we have that $k\geq 0$ and in this cases there exists at last one positive solution $x^*\in X$.
\end{remark}
\begin{remark} \label{var}Theorem \ref{existence} is also true under the same hypothesis for the following equation (replacing $G$ by $-G$):
$$(E) \hspace{3mm} x'(t) -G(x,t) x(t) = F(x,t), t\in \R.$$
To see this, just follow the same proof of Theorem \ref{existence} using  the operator $\tilde{T}$ instead of $T$, where :
\begin{eqnarray*}
\widetilde{T}: BC(\R) \times BC(\R)_{[l,+\infty[}&\to& BC(\R)\\
(f,g) &\mapsto& [t\mapsto \int_t^{+\infty} e^{\int_s^t g(u)du} f(s) ds],
\end{eqnarray*}
and the operator $\widetilde{\Gamma}(x):=\widetilde{T}(G(x),-F(x))$ for all $x\in X$, instead of $\Gamma$ in Lemma \ref{PAP}.
\end{remark}

\subsection{The proof of Theorem \ref{existence}}
In order to prove Theorem \ref{existence} we need some intermediate results, the proof will be given at the end of this section. Let us start with the following general lemmas.
\begin{lemma} \label{intG} Let $g \in BC(\R)$ such that  $\inf_{t\in \R} g(t)>0$. Then, we have that 
\begin{eqnarray*}
\int_{-\infty}^t  g(s)e^{- \int_s^t g(u) du}ds =1.
\end{eqnarray*}
\end{lemma}
\begin{proof} Since $\inf_{t\in \R} g(t) >0$, then clearly we have that $$\int_{-\infty}^t g(u) du=+\infty.$$ It follows that
\begin{eqnarray*}
\int_{-\infty}^t  g(s)e^{- \int_s^t g(u) du}ds&=& \int_{-\infty}^t  g(s)e^{ \int_t^s g(u) du}ds\\
                                                                         &=& [e^{ \int_t^s g(u) du}]_{-\infty}^t\\
                                                                                                        &=& 1- e^{- \int_{-\infty}^t g(u) du}\\
                                                                                                        &=& 1.
\end{eqnarray*}

\end{proof}
\begin{lemma} \label{Contint}  Let $l,r,r'>0$ be a strictly positive real numbers. Then, the following assertions hold. 

$(i)$ For every $(f,g)\in  BC(\R)\times BC(\R)_{[l,+\infty[}$, the function $T(f,g)$ is  Lipschitz on $\R$ with a constant of Lipschitz less than $\|f\|_{\infty}(\frac{\|g\|_{\infty}}{l}+1)$. Consequently, the familly $\mathcal{F}:=\lbrace T(f,g): (f,g)\in  B_{BC(\R)}(0,r) \times  BC(\R)_{[l,r']} \rbrace$ is uniformly equi-continuous on $\R$.

$(ii)$ The operator $T$ is Lipschitz on  $B_{BC(\R)}(0,r)\times BC(\R)_{[l,+\infty[}$, that is, for every $(f_1,g_1),(f_2,g_2)\in B_{BC(\R)}(0,r)\times BC(\R)_{[l,+\infty[}$, we have that
$$\|T(f_1,g_1) - T(f_2,g_2)\|_{\infty} \leq \max(\frac{r}{l^2},\frac{1}{l})(\|g_1-g_2\|_{\infty}+\|f_1-f_2\|_{\infty}).$$
\end{lemma}
\begin{proof} $(i)$ Let us prove that $T(f,g)$ is Lipschitz on $\R$. First, it is easy to see that since $\inf_{t\in \R} g(t) \geq l$, then
$$\|T(f,g)\|_{\infty}\leq \frac{\|f\|_{\infty}}{l}.$$
On the other hand, from the formula $(\ref{fixe})$, we have that for every $(f,g)\in BC(\R) \times BC(\R)_{[l,+\infty[}$:
\begin{eqnarray*}
T(f,g)'(t) &=&-g(t)T(f,g)(t)+f(t), \hspace{2mm} \forall t\in \R.  
\end{eqnarray*}
Thus, we have that $$\|T(f,g)'\|_{\infty}\leq \|g\|_{\infty}\|T(f,g)\|_{\infty}+\|f\|_{\infty}\leq \|f\|_{\infty}(\frac{\|g\|_{\infty}}{l}+1).$$
Hence, by the mean value theorem, $T(f,g)$ is  Lipschitz with a constant of Lipschitz less that $\|f\|_{\infty}(\frac{\|g\|_{\infty}}{l}+1)$. It follows that the familly $\mathcal{F}:=\lbrace T(f,g): (f,g)\in  B_{BC(\R)}(0,r) \times  BC(\R)_{[l,r']} \rbrace$ is unifomly equi-continuous on $\R$.

\vskip5mm
$(ii)$ First, recall that the function $x\mapsto e^x$ is $e^b$-Lipschitz on any intervalle $]-\infty,b]$, by the mean value theorem. Let $g_1, g_2\in BC(\R)_{[l,+\infty[}$.  Since $g_1, g_2\geq l$, we have that $ -\int_s^t g_1(u)du\leq -l(t-s)$ and $  -\int_s^t g_2(u)du\leq -l(t-s)$, for every $s\leq t$. It follows from the fact that $x\mapsto e^x$ is $e^{-l(t-s)}$-Lipschitz on the intervalle $]-\infty,-l(t-s)]$ that,
\begin{eqnarray*} \label{exp-lip}
|e^{-\int_s^t g_1(u)du}-e^{-\int_s^t g_2(u)du}|&\leq& e^{-l(t-s)}|\int_s^t g_1(u)du - \int_s^t g_2(u)du|.
\end{eqnarray*}
Thus, for every $t, s\in \R$ such that $s\leq t$, we have that
\begin{eqnarray*}
|e^{-\int_s^t g_1(u)du}-e^{-\int_s^t g_2(u)du}|&\leq& (t-s)e^{-l(t-s)}\|g_1-g_2\|_{\infty}.
\end{eqnarray*}
Now, using the above inequality we have that
\begin{eqnarray*}
|\int_{-\infty}^t e^{-\int_s^t g_1(u)du}f_1(s) ds-\int_{-\infty}^t e^{-\int_s^t g_2(u)du}f_2(s) ds|&\leq& \int_{-\infty}^t |e^{-\int_s^t g_1(u)du}-e^{-\int_s^t g_2(u)du} ||f_1(s)|ds\\
                                                                                                                                                        && + \int_{-\infty}^t e^{-\int_s^t g_2(u)du}|f_1(s)-f_2(s)|ds\\
                                                                            &\leq& r\|g_1-g_2\|_{\infty} \int_{-\infty}^t (t-s)e^{-l(t-s)}ds \\
                                                                            && +\|f_1-f_2\|_{\infty}\int_{-\infty}^t e^{-l(t-s)}ds\\
                                                                              &=& \frac{r}{l^2}\|g_1-g_2\|_{\infty}+\frac{1}{l}\|f_1-f_2\|_{\infty}.
\end{eqnarray*}
Thus, 
$$\|T(f_1,g_1)- T(f_1,g_1)\|_{\infty} \leq \max(\frac{r}{l^2}, \frac{1}{l}) (\|f_1-f_2\|_{\infty}+\|g_1-g_2\|_{\infty}).$$
This ends the proof of $(ii)$.
\end{proof} 



\begin{lemma} \label{PAP} Under the hypothesis $(H_0)$ and $(H_1)$ (resp. the hypothesis $(H_0)$ and $(\tilde{H}_1)$), the operators $\Gamma$ defined for all $x\in B_{[k,M]}\subset X$ by 
\begin{eqnarray*}
\Gamma(x):=T(F(x),G(x))=[t\mapsto \int_{-\infty}^t  e^{-\int_s^t G(x,u)du} F(x,s)ds],
\end{eqnarray*}
satisfies the following assertions:

$(a)$ $\Gamma$ maps $B_{[k,M]}$ into $B_{[k,M]}$.

$(b)$  $\Gamma$  is   norm-to-norm continuous and satisfies: for every sequence $(x_n)\subset B_{[k,M]}$, if $(x_n)$ converges uniformly on each compact of $\R$ to some point of $BC(\R)$, then $(\Gamma(x_n))$ is relatively compact in $(B_{[k,M]},\|\cdot\|_{\infty})$ (resp. $\Gamma$ is $\|\cdot\|_{\infty}$-to-$\|\cdot\|_{\infty}$ contractant).

$(c)$ $\Gamma(B_{[k,M]})$ is equi-continuous at each point of $\R$. Moreover, the  set $\Gamma(B_{[k,M]})(t):=\lbrace \Gamma (x)(t) : x\in B_{[k,M]} \rbrace \subset [k,M]$ is relatively compact in $\R$.
\end{lemma}
\begin{proof} Assume $(H_0)$ and $(H_1)$  and let us set $r:=\sup_{x\in B_{[k,M]}}\|F(x)\|_{\infty}$ and $l:=\inf_{x\in X, t\in \R}G(x,t) >0$. 
Since $T$ satisfies $(H_0)$,  the operator $\Gamma$ maps $B_{[k,M]}$ into $X$ as follows: 
\begin{eqnarray*}
\Gamma: B_{[k,M]} &\to& B_X(0,r)\times X_{[l,+\infty[}\to X\\
                    x &\mapsto& (F(x),G(x)) \mapsto \Gamma(x)=T(F(x),G(x)).
\end{eqnarray*}

$(a)$ Let us prove that $\Gamma$ maps $B_{[k,M]}$ into $B_{[k,M]}$. Indeed,  by assumption we have that
\begin{eqnarray*}
k \leq \frac{F(x,s)}{G(x,s)}\leq M, \forall (x,t)\in B_{[k,M]}\times \R,
\end{eqnarray*}
we get using Lemma \ref{intG} that for every $x\in B_{[k,M]}$ and every $t\in \R$
\begin{eqnarray*}
k &=&\int_{-\infty}^t  k G(x)(s) e^{-\int_s^t G(x,u)du} ds \\
&\leq& \Gamma(x)(t)= \int_{-\infty}^t  e^{-\int_s^t G(x,u)du} F(x,s)ds\\ 
&\leq& \int_{-\infty}^t  M G(x)(s)e^{-\int_s^t G(x,u)du} ds\\
&=&M.
\end{eqnarray*}
Thus, $\Gamma(x)\in B_{[k,M]}$.  

$(b)$ Using part $(ii)$ of lemma \ref{Contint}, we get that for every $x,y\in B_{[k,M]}$, 
\begin{eqnarray}\label{EQG}
\|\Gamma(x)-\Gamma(y)\|_{\infty} &=& \|T(F(x),G(x))-T(F(y),G(y))\|_{\infty}\nonumber\\
&\leq& \max(\frac{r}{l^2}, \frac{1}{l}) (\|F(x)-F(y)\|_{\infty}+\|G(x)-G(y)\|_{\infty}).
\end{eqnarray}
It follows using the hypothesis $(H_1)$, that $\Gamma$ is continuous and that for every sequence $(x_n)\subset B_{[k,M]}$, if $(x_n)$ converges on each compact subset of $\R$ to some point of $BC(\R)$, then $(\Gamma(x_n))$ is relatively compact in $(B_{[k,M]},\|\cdot\|_{\infty})$.

$(c)$  We obtain that $\Gamma(B_{[k,M]}) (\subset  B_{[k,M]})$ is  equi-continuous at each point of $\R$ by using Lemma \ref{Contint} with $r:=\sup_{x\in B_{[k,M]}}\|F(x)\|_{\infty}$, $l:=\inf_{x\in X, t\in \R}G(x,t) >0$ and $r':=\sup_{x\in B_{[k,M]}, t\in \R} G(x,t)>0$. Moreover, it is clear that the  set $\Gamma(B_{[k,M]})(t):=\lbrace \Gamma (x)(t) : x\in B_{[k,M]} \rbrace \subset [k,M]$ is relatively compact in $\R$.

Now, if we assume that the hypothesis $(\tilde{H}_1)$ holds, then $F$ and $G$ are $\|\cdot\|_{\infty}$-to-$\|\cdot\|_{\infty}$ Lipschitz  functions and so using the inequality $(\ref{EQG})$ we get that
$$\|\Gamma(x)-\Gamma(y)\|_{\infty} \leq \max(\frac{r}{l^2}, \frac{1}{l})(L_F+L_G)\|x-y\|_{\infty},$$
which implies that $\Gamma$ is contractant by the assumption $(\tilde{H}_1)$. 

\end{proof}


Now, we give the proof of Theorem \ref{existence}. Let us denote, $\overline{\textnormal{co}}^{\|\cdot\|_{\infty}}(\Gamma(B_{[k,M]}))$  the norm-closed convex hull of $\Gamma(B_{[k,M]})$.

\begin{proof}[Proof of Theorem \ref{existence}] We treat two situations:

$\bullet$ {\it Under the hypothesis $(H_0)$ and $(H_1)$.}  By Lemma \ref{PAP}, $\Gamma(B_{[k,M]})\subset B_{[k,M]}$ and so $$K:=\overline{\textnormal{co}}^{\|\cdot\|_{\infty}}(\Gamma(B_{[k,M]}))\subset B_{[k,M]}.$$ 
Then, $\Gamma(K)\subset \Gamma(B_{[k,M]})\subset K$ and  we have that the operator $\Gamma: K \to K$ is well defined and continuous. We are going to prove that  $\Gamma(K)$ is relatively compact for the norm $\|\cdot\|_{\infty}$.  Using part $(c)$ of Lemma \ref{PAP}, we see that $K$ (as a closed convex hull)  is also equi-continuous at each point of $\R$ and that $K(t):=\lbrace x(t): x\in K \rbrace \subset [k,M]$ is relatively compact in $\R$. Thus, from the Arzela-Ascoli theorem, we have that the restriction of $K$ to any intervalle $[-m,m]$ of $\R$ ($m\in \N$)  is relatively compact  in the space $(C([-m,m]),\|\cdot\|_{\infty})$ of all continuous functions on $[-m,m]$.  Now,  let $(x_n)$ be any sequence of $K$. Then, we have that the restriction of $(x_n)$ to each intervalle $[-m,m]$ has a subsequence $(x_{\mu_m(n)})$ converging uniformly  on this intervalle. Using the Cantor diagonal process, there exists a subsequence $(x_{\mu(n)})$ converging uniformly on each compact subset of $\R$.  Then, by Lemma \ref{PAP}, there exists a subsequence that we will denote again $(x_{\mu(n)})$ such that $(\Gamma(x_{\mu(n)}))$  norm converges in $BC(\R)$. Thus, $\Gamma(K)$ is relatively compact for the norm $\|\cdot\|_{\infty}$. Using the Schauder fixed point theorem we get a fixed point $x^*\in K\subset B_{[k,M]}$ for $\Gamma$, which satisfies the equation $(E)$ by the formula (\ref{fixe}). 

$\bullet$ {\it Under the hypothesis $(H_0)$ and $(\tilde{H}_1)$.}  In this situation the operator $\Gamma$ is contractant by Lemma \ref{PAP}, so the Banach-Picard theorem applies and gives a unique fixed point $x^*\in  B_{[k,M]}$ for $\Gamma$, which is the unique solution of the equation $(E)$ in the set $B_{[k,M]}$ by the formula (\ref{fixe}). 
\end{proof}
\vskip5mm
\subsection{Examples and properties} \label{Sexemple}
In this section, we give examples satisfying our results. The hypothesis $(H_0)$ is satified for several classical subspace of $BC(\R)$. We give in Proposition \ref{algebra} (see bellow) some examples of classical spaces satisfying this property. We need to introduce some definitions.
\vskip5mm
 For a fixed $w\in \R$, we denote $C_w(\R)$ the Banach subspace of $BC(\R)$ consisting on  all continuous $w$-periodic functions. 

\begin{definition} A continuous function $f : \R \to \R$ is called (Bohr) almost periodic if for each
$\varepsilon > 0$, there exists $l_\varepsilon > 0$ such that every interval of length $l_\varepsilon$ contains at least a
number $\tau$ with the following property:
$$\sup_{t\in \R}|f(t+\tau)-f(t)|< \varepsilon.$$
\end{definition}
The number $\tau$ is then called an $\varepsilon$-period of $f$. The collection of all almost periodic functions $f : \R\to \R$ will be denoted by $AP(\R)$. It is known that the space $AP(\R)$ is a Banach subspace of $BC(\R)$ (see for instance \cite{Dt}). Clearly, for every $w\in \R$, we have that $C_w(\R)\hookrightarrow AP(\R)$ (a Banach subspace). A classical example of an almost periodic function which is not periodic is given
by the following function
$$f (t) = sin t +sin\sqrt{2}t.$$
 The space of continuous ergodic functions is defined as follows:
$$PAP_0(\R):=\lbrace g\in BC(\R): \lim_{r\to +\infty} \frac{1}{2r}\int_{-r}^{r} |g(t)| dt=0\rbrace.$$
Clearly, $PAP_0(\R)$ is a Banach subspace of $BC(\R)$. It is easy to see that $AP(\R)\cap PAP_0(\R)=\lbrace 0 \rbrace$ (see for instace \cite{Dt}). Then, we define the Banach subspace of $BC(\R)$ of all pseudo almost periodic function denoted $PAP(\R)$, as follows:
$$PAP(\R):=AP(\R)\oplus PAP_0(\R).$$
Finally, we introduce the following space of all pseudo $w$-periodic functions denoted $PP_w(\R)$ by 
$$PP_w(\R):=C_w(\R)\oplus PAP_0(\R).$$
Clearly, $PP_w(\R)$ is a Banach subspace of $PAP(\R)$ for every $w\in \R$. Finally, $BC_U(\R)$ denotes the Banach subspace of $BC(\R)$ of uniformly continuous functions.

\begin{proposition} \label{algebra} The following classical spaces  $X:=BC(\R)$, $BC_U(\R)$, $C_w(\R)$, $AP(\R)$, $PAP_0(\R)$, $PAP(\R)$ and $PP_w(\R)$, satisfy the hypothesis $(H_0)$.
\end{proposition}
\begin{proof} The result is clear and easy  for $X=BC(\R)$, $C_w(\R)$. For $X=BC_U(\R)$, we use the point $(i)$ of Lemma \ref{Contint}.  The proof for $X=AP(\R)$  can be found in \cite[Lemma 1.3]{DLN}.  For $X=PAP_0(\R)$, just follow the proof of \cite[Lemma 1.3]{Chz}. The proof for $X=PAP(\R)$ is given in step 2 of the proof of \cite[Theorem 1]{BC} and finally the proof  for $X=PP_w(\R)$ can be given in the same way. 
\end{proof}

Given two functions $f : \R\to \R$ and $g :\R \to \R$, their convolution $f *g$, if it exists, is defined by
$$f*g(t)=\int_{-\infty}^{+\infty} f(s)g(s-t) ds.$$
One can generate various types of almost periodic functions using the convolution.
\begin{proposition} $($see \cite[Proposition 3.4 and Proposition 5.3]{Dt}$)$ \label{conv} The following assertions hold.

$(i)$ Let $x\in AP(\R)$ and $\alpha\in L^1(\R)$. Then $x* \alpha\in AP(\R)$.

$(ii)$ Let $x\in PAP_0(\R)$ and $\alpha\in L^1(\R)$. Then $x* \alpha\in PAP_0(\R)$.

$(iii)$  Let $x\in PAP(\R)$ and $\alpha\in L^1(\R)$. Then $x* \alpha\in PAP(\R)$.
\end{proposition}

Now,  we give a general way to construct Lipschitz functions $F: PAP(\R) \to PAP(\R)$. 

\begin{definition} $($see \cite{Chz}$)$  Let $\Omega\subset \R$. A continuous function $f: \Omega \times \R \to \R$ is called pseudo almost periodic in $t$ uniformly with respect $x\in \Omega$, if the two following conditions are satisfied:

$(i)$  $\forall x\in \Omega$, $f(x,\cdot)\in PAP(\R)$.

$(ii)$ for all compact set $K\subset \Omega$, we have that: $\forall \varepsilon >0$, $\exists \delta>0$, $\forall t\in \R$, $\forall x, y \in K$:
$$|x-y|\leq \delta \Longrightarrow |f(x,t)-f(y,t)|\leq \varepsilon.$$
The set of all such functions will be denoted $PAP_U(\\Omega\times R)$.
\end{definition}

\begin{lemma} $($see \cite{CieuFG}$)$ Let $f\in PAP_U(\R)$ and $x\in PAP(\R)$. Suppose that bounded subset $B$ of $\R$, $f$ is bounded on $B\times \R$. Then, the function $[t\mapsto f(x(t),t)]\in PAP(\R)$.
\end{lemma}
Using the above lemma, we deduce easily the following proposition.
\begin{proposition} \label{pap_u} Let $f\in PAP_U(\R)$ be a Lipschitz function with respect to the first variable, that is, there exists $L_f\geq 0$ such that:
$$|f(s,t)-f(s',t)|\leq L_f|s-s'|, \hspace{2mm} \forall s, s', t\in \R.$$
Then, the function $F$ defined by $F(x):=f(x(\cdot),\cdot)$ is a Lipschitz function for the norm $\|\cdot\|_{\infty}$  that maps $(PAP(\R), \|\cdot\|_{\infty})$ into  $(PAP(\R), \|\cdot\|_{\infty})$.
\end{proposition}
The above proposition says that each Lipschitz function $f\in PAP_U(\R)$ induce a Lipschitz function $F: (PAP(\R), \|\cdot\|_{\infty}) \to (PAP(\R), \|\cdot\|_{\infty})$. However, the converse is not true in general. Indeed, for any fixed $\alpha\in L^1(\R)\setminus \lbrace 0 \rbrace$, the $\|\alpha\|_1$-Lipschitz function $F:  (PAP(\R), \|\cdot\|_{\infty}) \to (PAP(\R), \|\cdot\|_{\infty})$, defined by $$F(x,t):=sin(t)+sin(\sqrt{2}t)+ (1+t^2)^{-1}x*\alpha(t)$$ cannot, under any circumstances, be written in the form $f(x(t),t)$ for some $f\in PAP_U(\R)$. This prove  that there are many more  continuous (Lipschitz) functions $F:  (PAP(\R), \|\cdot\|_{\infty}) \to (PAP(\R), \|\cdot\|_{\infty})$, than those which come from the functions $f\in PAP_U(\R)$ as in Proposition \ref{pap_u}.
\vskip5mm

Now, we give simple examples satisfying our theorems. We start with the example announced in the introduction.

\begin{Exemp} \label{Ex0}
Let $\alpha, \beta\in L^1(\R)$ be such that $0<\|\alpha\|_1\leq 1$. 
\begin{eqnarray*}
F(x,t)&=&\frac{1}{3}(sin(t)+sin(\sqrt{2}t)+ (1+t^2)^{-1}x*\alpha(t))\\
G(x,t)&=&3+sin(2t)+(1+t^2)^{-1}cos(x*\beta(t)).
\end{eqnarray*}
The hypothesis $(H_0)$ is satisfied by Proposition \ref{algebra}. We are going to prove that the hypothesis $(H_1)$ is also satisfied. Indeed, clearly, $F, G: PAP(\R)\to PAP(\R)$ are $\|\alpha\|_1$-Lipschitz and $\|\beta\|_1$-Lipschitz respectively (Notice that $F$ is not bounded but $F$ and $G$ are bounded on bounded sets).  We have that $G(x,t)\geq 1$, for every $(x,t)\in PAP(\R)\times \R$. On the other hand, we have that $|F(x,t)|\leq 2+\|x\|_{\infty}\|\alpha\|_1$ for every $(x,t)\in PAP(\R)\times \R$. Then, for every $x\in  PAP(\R)$ such that $\|x\|_{\infty}\leq \frac{1}{\|\alpha\|_1}$, we have that 
\begin{eqnarray*}
\left|\frac{F(x,t)}{G(x,t)}\right|&\leq& \frac{1}{3}(2+\|x\|_{\infty}\|\alpha\|_1)\\
&\leq &1\\
&\leq& \frac{1}{\|\alpha\|_1}.
\end{eqnarray*}
Thus, 
$$-\frac{1}{\|\alpha\|_1}\leq \frac{F(x,t)}{G(x,t)}\leq \frac{1}{\|\alpha\|_1}, \hspace{2mm} \forall (x,t)\in B_{[-\frac{1}{\|\alpha\|_1},\frac{1}{\|\alpha\|_1}]}\times \R.$$
Now, let $(x_n)\subset PAP(\R)$ be a sequence converging on each compact of $\R$ to some $x\in BC(\R)$. Then, there exists a constant $M\geq \|x\|_{\infty}$ such that $\|x_n\|_{\infty}\leq M$ for all $n\in \N$. Then, for every $\varepsilon>0$, there exists $A_\varepsilon>0$ such that for every $|t| \geq A_\varepsilon$, we have that $(1+t^2)^{-1}\leq \frac{\varepsilon}{2M\|\alpha\|_1}$. Thus, 
\begin{eqnarray*}
(1+t^2)^{-1}|x_n*\alpha(t)-x*\alpha(t)|&\leq& 2M\|\alpha\|_1(1+t^2)^{-1}<\varepsilon, \hspace{2mm} \forall |t| \geq A_\varepsilon.
\end{eqnarray*}
On the other hand, there exists $N\in \N$ such that for every $n\geq N$, we have that $\sup_{t\in [-A_\varepsilon,A_\varepsilon]}|x_n(t)-x(t)|<\varepsilon$. Hence, 
\begin{eqnarray*}
\sup_{t\in \R} (1+t^2)^{-1}|x_n*\alpha(t)-x*\alpha(t)|&\leq& 2\varepsilon, \hspace{2mm} \forall n\geq N.
\end{eqnarray*}
Hence, we get that the sequence $(F(x_n))$ norm converges in $BC(\R)$ and so it is relatively compact in $PAP(\R)$. The same argument hold for $(G(x_n))$. Thus, the hypothesis $(H_0)$ and $(H_1)$ are satisfied, so Theorem \ref{existence} gives at last one solution $x^*\in PAP(\R)$ for the equation $(E)$. 
\end{Exemp}

\begin{Exemp} \label{Ex1} ({\bf Example in the space $PP_w(\R)\subset PAP(\R)$ under the hypothesis $(H_1)$}) As a simple consequence of Theorem \ref{existence},  we obtain that  the following equation has a positive solution $x^*\in PAP(\R)$
\begin{eqnarray*}
x'(t)+G(x,t) x(t)&=& F(x,t)
\end{eqnarray*}
where $F, G:  PP_w(\R) \to PP_w(\R)$ are defined by $G(x,t)=e^{\|x\|_{0,1}+\frac{1}{1+t^2}sin(x(t))}$ and  $F(x,t):=3+sin(\|x\|_{0,1} )+\frac{1}{1+t^2}cos(x(t))$, where  $\|x\|_{0,1}:=\int_0^1 |x(s)|ds\leq \|x\|_{\infty}$. 

$\bullet$ The space $PP_w(\R)$ satisfies $(H_0)$ by Proposition \ref{algebra}.

$\bullet$ The hypothesis $(H_1)$ is satisfied. Indeed, clearly we have that $$\inf_{x\in X, t\in \R} G(x,t)\geq e^{-1}>0.$$ Let us set $G_0(x,t):=\|x\|_{0,1}+\frac{1}{1+t^2}sin(x(t))$, we have that
\begin{eqnarray*}
|G_0(x,t)-G_0(y,t)|&\leq& \|x-y\|_{0,1}+\frac{1}{1+t^2}|sin(x(t))-sin(y(t))|\\
&\leq& \|x-y\|_{0,1}+\frac{1}{1+t^2}|x(t)-y(t)|\\
&\leq& 2\|x-y\|_{\infty}.
\end{eqnarray*}
Then, $G_0$ is Lipschitz from $PP_w(\R)$ into $PP_w(\R)$. Since $t\mapsto e^t$ is continuous, it follows that $G$ is continuous from $PP_w(\R)$ into $PP_w(\R)$. On the other hand, it is clear that $G$ is bounded on bounded sets. Finally, let $(x_n)\subset X$ be a sequence  converging on each compact of $\R$.  Notice that $|\|x_n\|_{0,1}-\|x_m\|_{0,1}|\leq \|x_n-x_m\|_{0,1}\leq \sup_{t\in [0,1]}|x_n(t)-x_m(t)|$. On the other hand, we see that for every $A\geq 1$, we have that
\begin{eqnarray*}
\sup_{t\in \R}|\frac{1}{1+t^2}sin(x_n(t))-\frac{1}{1+t^2}sin(x_m(t))|&\leq&\\
\max(\sup_{t\in [-A,A]}|\frac{1}{1+t^2}(sin(x_n(t))-sin(x_m(t)))|,\frac{2}{1+A^2})&\leq&\\
\max(\sup_{t\in [-A,A]}|x_n(t)-x_m(t)|,\frac{2}{1+A^2}).
\end{eqnarray*}
It follows that 
\begin{eqnarray*}
\|G_0(x_n)-G_0(x_m)\|_{\infty} \leq \sup_{t\in [0,1]}|x_n(t)-x_m(t)| +\max(\sup_{t\in [-A,A]}|x_n(t)-x_m(t)|,\frac{2}{1+A^2}).
\end{eqnarray*}
So, $\limsup_{n,m\to+\infty} \|G_0(x_n)-G_0(x_m)\|_{\infty}\leq \frac{2}{1+A^2}$ for every $A \geq 1$. Sending $A$ to $+\infty$, we get that $(G_0(x_n))$ norm converges and by the continuity of the function $t\mapsto e^t$, we have that $(G(x_n))$ norm converges.

 As above, we see that $(F(x_n))$  norm converges. Now, it is clear $0\leq F(x,t)\leq 5$ for every $x\in PP_w(\R)$ and every $t\in \R$. Thus, we can take 
$$0\leq k:=  \inf_{(x,t)\in X\times \R} \frac{F(x,t)}{G(x,t)} \textnormal{ and } M:= \sup_{(x,t)\in X\times \R} \frac{F(x,t)}{G(x,t)},$$
and so the hypothesis $(H_1)$ is satisfied. Using Theorem \ref{existence} we get that there exists at last one positive solution of the equation $(E)$.

\end{Exemp}

\begin{Exemp} ({\bf Example in the space  $C_{2\pi}(\R)$ under the hypothesis $(\tilde{H}_1)$}.) We have the following simple example.

$\bullet$ $G:  C_{2\pi}(\R) \to C_{2\pi}(\R)$, defined by $$G(x,t)=4+(1+\|x\|_{0,1})(1+sin(t)),$$
 is a $2$-Lipschitz (on the variable $x$)  and satisfies $\inf_{x\in X, t\in \R} G(x,t)\geq l=4$ (where  $\|x\|_{0,1}:=\int_0^1 |x(s)|ds\leq \|x\|_{\infty}$)

$\bullet$ $F:  C_{2\pi}(\R) \to C_{2\pi}(\R)$, defined by $F(x,t):=2+sin(t)+cos(x(t))$, is $1$-Lipschitz (on the variable $x$) and  is bounded by $r=4$.

Since $\max(\frac{r}{l^2}, \frac{1}{l})(L_F+L_G)= \frac{3}{4}<1,$ then $(\tilde{H}_1)$ is satisfied. This permit to give thanks to Theorem \ref{existence} a  continuous $2\pi$-periodic solution to the equation $(E)$. Moreover, there exists a unique solution in $B_{[0,1/2]}$.
\end{Exemp}
\section{Global attractivity of solutions} \label{Global}
In this section, we give a result on the attractivity of solutions of the equation $(E)$. 

\begin{definition} A solution $x^*$ of the equation $(E)$, is said to be globally attractive if for any other solution $x$ of $(E)$, we have that $\lim_{t\to +\infty} |x(t)-x^*(t)|=0$.
\end{definition}
Consider the following condition:

$(C)$ the functions $F, G :X\to X$ satisfies $(H_1)$ and moreover: there exists $L_F, L_G \geq 0$ such that 

$(i)$ $|F(x,t)-F(y,t)|\leq L_F |x(t)-y(t)|$ and $|G(x,t)-G(y,t)|\leq L_G |x(t)-y(t)|$, $\forall x, y \in X$ and $t\in \R$.

$(ii)$ $l> L_{G}\max(M,-k) +L_{F},$ where $l>0$, $k$ and $M$ ($k\leq M$) are given by the hypothesis $(H_1)$.

Notice that general examples of functions satisfying the point $(i)$ above, are given by Proposition \ref{pap_u} (see also Example \ref{Exatt}).
\begin{theorem} \label{attractivity} Under the hypothesis $(H_0)$ and $(C)$, there exists at last one solution $x^*\in X$ for the equation $(E)$, which is globally attractive.
\end{theorem}

\begin{proof} By Theorem \ref{existence}, there exists at last one solution $x^*\in B_{[k,M]}$, this implies that  $\|x^*\|_{\infty}\leq \max(M,-k)$.  Suppose that $x$ is another solution of $(\tilde{E})$.  In this case, we have that for every $t\in \R$
$$x'(t)=-G(x,t)x(t)+F(x,t),$$
$$(x^*)'(t)=-G(x^*,t)x^*(t)+F(x^*,t).$$

It follows that $|x'(t)|\leq \|G(x)\|_{\infty} \|x\|_{\infty}+\|F(x)\|_{\infty}<+\infty$ for every $t\in \R$ and so $x$ is a Lipschitz function on $\R$ by the mean value theorem. Similarily, $x^*$ is a Lipschitz function on $\R$.  Let us denote $sgn(t)$ the number which is equal to $1$ if $t\geq 0$ and $-1$ otherwize. Consider the following Lyapunov functional:
$$ W(t)=|x(t)-x^*(t)|, \hspace{3mm} \forall t\in \R.$$
After calculating the Dini derivative of $W$, we get, 
\begin{eqnarray*}
D^+W(t) &=&  sgn(x(t)-x^*(t))\left (x'(t))-(x^*)'(t)\right)\\
&=&  sgn(x(t)-x^*(t))[ -G(x,t)x(t)+G(x^*,t)x^*(t) + F(x,t)-F(x^*,t)]\\
&=&  sgn(x(t)-x^*(t))[ -(G(x,t)(x(t)-x^*(t))+\\
&&(G(x^*,t)-G(x,t))x^*(t) + F(x,t)-F(x^*,t)]\\
&=& -G(x,t)|x(t)-x^*(t)|+ sgn(x(t)-x^*(t))[G(x^*,t)-G(x,t)]x^*(t) + \\
&& sgn(x(t)-x^*(t))[F(x,t)-F(x^*,t)]\\
&\leq& -l|x(t)-x^*(t)|+(L_{G}\|x^*\|_{\infty} +L_{F} )|x(t)-x^*(t)|\\
&\leq& -(l-(L_{G}\max(M,-k) +L_{F})) |x(t)-x^*(t)|.
\end{eqnarray*}
An integration of the above inequality, gives  
\begin{eqnarray*}
 W(t) +\int_0^t  (l-(L_{G}\max(M,-k) +L_{F}))|x(t)-x^*(t)| ds &\leq& W(0).
\end{eqnarray*}
Since $W(t) \geq 0$ for every $t \in \R$, it follows that  
\begin{eqnarray*} 
\limsup_{t\to +\infty} \int_0^t  |x(t)-x^*(t)| ds &\leq& \frac{W(0)}{l-(L_{G}\max(M,-k)+L_{F})}<+\infty.
\end{eqnarray*}
Since $x, x^* \in X$ are uniformly continuous on $\R$ (in fact Lipschitz), we obtain  that  $\lim_{t\to +\infty} |x(t)-x^*(t)|=0.$
\end{proof}
\begin{Exemp} \label{Exatt} ({\bf Example of globally attractive solution in the space $PAP(\R)$ under the hypothesis $(C)$}.) 

$\bullet$ $G:  PAP(\R) \to PAP(\R)$, defined by $$G(x,t)=4+sin(t)+sin(\sqrt{2}t) +\frac{1}{1+t^2}cos(x(t)),$$ is $1$-Lipschitz (on the variable $x$) and $l:=\inf_{x\in X, t\in \R} G(x,t)\geq 1>0.$

$\bullet$ $F:  PAP(\R) \to PAP(\R)$, defined by $$F(x,t):=\frac{1}{10}[2+cos(t)+\frac{1}{1+t^2}sin(x(t))],$$ is $1/10$-Lipschitz (on the variable $x$) and $0\leq F(x,t)\leq 4/10$.

We have that $$ 0=k\leq \frac{F(x,t)}{G(x,t)}\leq M=4/10.$$

As in the Example \ref{Ex1} the hypothesis $(H_1)$ is satisfied. Since $\max(M,-k)=4/10$, $L_G\leq 1$, $L_F\leq 1/10$ we have $l\geq 1> 4/10+ 1/10\geq \max(M,-k)L_G+L_F,$ then the condition $(C)$ is satisfied and so by Theorem \ref{attractivity} there exists at last one positive solution in $PAP(\R)$ which is globally attractive.  
\end{Exemp}
\section{The vector valued framework} \label{Vector}
The results of  existence of solutions developped in Section \ref{S1} can be easily extended to the vector-valued framework. We will just present the outline of the procedure to follow, since the proofs are similar to those given in Section \ref{S1}.

Let $(E,\|\cdot\|)$ be a finite dimensional Banach space (for example $E=\R^n$) and  $E^*$ its dual. Let $X$ be a Banach subspace of the Banach space  $(BC(\R,E),\|\cdot\|_{\infty})$  of all $E$-valued bounded continuous  functions equipped with the sup-norm and  let $e^*\in E^*$ and $F, G :X\to X$  be two functions. The goal of this section to give existence of solutions $x\in X$  to the first order differential equations of type (under conditions similar to those given in Section \ref{S1}):
$$(\widetilde{E})  \hspace{3mm} x'(t) +e^*(G(x,t)) x(t) = F(x,t), \hspace{2mm} \forall t\in \R.$$
By $B_X(0,c)$ we denote the closed ball of $X$ centred at $0$ with radious $c>0$. For each $l,r>0$ and $e^*\in E^*$, we define the following closed convex subsets of $X$:  
$$X_{[l,+\infty[,e^*}:=\lbrace x\in X: e^*(x(t))\in [l,+\infty[, \forall t\in \R\rbrace,$$
As in Section \ref{S1}, we introduce the following well defined operator, for each $l>0$ and each $e^*\in E^*$:
\begin{eqnarray*}
T: BC(\R,E) \times BC(\R,E)_{[l,+\infty[,e^*}&\to& BC(\R,E)\\
(f,g) &\mapsto& [t\mapsto \int_{-\infty}^t e^{-\int_s^t e^*(g(u))du} f(s) ds],
\end{eqnarray*}
and we have that for every $(f,g)\in  BC(\R,E)\times BC(\R,E)_{[l,+\infty[,e^*}$, the function $T(f,g)$ is differentiable and satisfies:
\begin{eqnarray*}
T(f,g)'(t) &=&-e^*(g(t))T(f,g)(t)+f(t), \hspace{2mm} \forall t\in \R.
\end{eqnarray*}

\vskip5mm

We consider  the following conditions $(HV_0)$, $(HV_1)$ and $(\widetilde{HV}_1)$:

$(HV_0)$ the subspace $X$ is invariant under $T$ in the sens that for each $l>0$ and every $e^*\in E^*$,  $T(X \times X_{[l,+\infty[,e^*})\subset X$.

$(HV_1)$ The functions  $F, G :(X,\|\cdot\|_{\infty}) \to (X,\|\cdot\|_{\infty})$ are continuous, $e^*\in E^*$ and : 

$\bullet$ $\inf_{x\in X, t\in \R}e^*(G(x,t)) >0$ and  there exists $c>0$ such that $F,G$ are bounded on $B_X(0,c)$ and $\frac{\|F(x,t)\|}{e^*(G(x,t))}\leq c$ for all $(x,t)\in B_X(0,c)\times \R$.

$\bullet$ for every sequence $(x_n)\subset X$, if $(x_n)$ converges on each compact of $\R$ in $BC(\R,E)$, then $(F(x_n))$ and $(G(x_n))$ are relatively compact in $(X,\|\cdot\|_{\infty})$.
\vskip5mm
 Notice that if we assume that $F$ and $G$ are bounded on the whole space $X$ and $\inf_{x\in X, t\in \R}e^*(G(x,t)) >0,$ then we can take  in the hypothesis $(HV_1)$
$$ c:= \sup_{(x,t)\in X\times \R} \frac{\|F(x,t)\|}{e^*(G(x,t))}.$$

$(\widetilde{HV}_1)$ The functions  $F, G :(X,\|\cdot\|_{\infty}) \to (X,\|\cdot\|_{\infty})$ are Lipschitz, $$l:=\inf_{x\in X, t\in \R}e^*(G(x,t)) >0,$$
and there exists $c>0$ such that, $\frac{\|F(x,t)\|}{e^*(G(x,t))}\leq c$ for all $(x,t)\in B_X(0,c)\times \R$, $$r:=\sup_{x\in B_C(0,c)}\|F(x)\|_{\infty}<+\infty$$
 and  $$\max(\frac{r}{l^2}, \frac{1}{l})(L_F+L_G)<1,$$
where $L_F$ and $L_G$ denotes the constant of Lipschitz of $F$ and $G$ respectively.
\vskip5mm

\begin{theorem}  Under the hypothesis $(HV_0)$ and $(HV_1)$ (resp. $(HV_0)$ and $(\widetilde{HV}_1)$), the equation $(EV)$  has at least one  solution $x^*$ in $B_X(0,c)$  (resp. has a unique solution $x^*$ in $B_X(0,c)$).
\end{theorem}
The proof of the above theorem is similar to those given for Theorem \ref{existence}. 
\vskip5mm
In the case of $X=PAP(\R,\R^n)$ for example (see \cite{Dt}, for the definition of vector-valued pseudo almost periodic), there exists  a canonical way to choose a function $G$ satisfying the property $(\widetilde{HV}_1)$ from the real-valued framework. It suffises  to choose functions $H: PAP(\R,\R^n)\to PAP(\R)$, and $(e,e^*)\in \R^n\times (\R^n)^*$ such that $e^*(e)=1$ and pose $G: PAP(\R,\R^n)\to PAP(\R,\R^n)$ by setting $G(x,t):=H(x,t)e$. 

\begin{remark} The above theorem  is also true under the same hypothesis for the following equation (replacing $G$ by $-G$):
$$(\widetilde{E}) \hspace{3mm} x'(t) -e^*(G(x,t)) x(t) = F(x,t), t\in \R.$$
To see this, just follow the same proof of Theorem \ref{existence} using  the operator $\tilde{T}$ instead of $T$, where :
\begin{eqnarray*}
\tilde{T}: BC(\R,E) \times BC(\R,E)_{[l,+\infty[,e^*}&\to& BC(\R,E)\\
(f,g) &\mapsto& [t\mapsto \int_t^{+\infty} e^{\int_s^t e^*(g(u))du} f(s) ds].
\end{eqnarray*}
and the operator $\widetilde{\Gamma}(x):=\widetilde{T}(e^*\circ G(x),-F(x))$ for all $x\in X$, instead of $\Gamma$ in Lemma \ref{PAP}.
\end{remark}

\section*{Acknowledgment}
This research has been conducted within the FP2M federation (CNRS FR 2036) and  SAMM Laboratory of the University Paris Panthéon-Sorbonne. 

\bibliographystyle{amsplain}

\end{document}